\documentclass[a4paper]{article}
\usepackage[english]{babel}
\usepackage[cp1251]{inputenc}
\pdfoutput=1
\usepackage[colorlinks,linkcolor=blue,citecolor=blue,urlcolor=blue]{hyperref}
\usepackage{latexsym,amsfonts,amssymb,amsthm,amsmath,amsbsy,graphicx}
\newtheorem{cor}{Corollary}[section]
\newtheorem{prp}{Proposition}[section]
\theoremstyle{definition}
\newtheorem{exm}{Example}[section]
\begin{document}
\thispagestyle{empty}
\bigskip
 {\noindent \Large\bf\sc  the unified
form of pollaczek--khinchine formula for l\'evy processes with matrix-exponential\\
negative jumps}\footnotemark[1] \footnotetext{This is an
electronic version of the original article published in Theory of Stochastic Processes, 
Vol.~18(34), no.~2, 2012.}

\bigskip
 { \bf D. Gusak,\footnote{Institute of Mathematics, Ukrainian National
 Academy of Science, 3 Tereshenkivska str.,
01601 Kyiv, Ukraine. \href{mailto:gusak\_dv@mail.ru}{gusak\_dv@mail.ru}}
Ie.~Karnaukh
 \footnote{Department of Statistics and Probability Theory,
O. Honchar Dnipropetrovsk National University, 72, Gagarina Pr.,
Dnipropetrovsk 49010, Ukraine.
 \href{mailto:ievgen.karnaukh@gmail.com}{ievgen.karnaukh@gmail.com}
 }}
\begin{center}
\bigskip
\begin{quotation}
\noindent  {\small For L\'evy processes with matrix-exponential negative jumps, the unified
form of the Pollaczek-Khinchine formula is established.}
\end{quotation}
\end{center}
\bigskip

\section{Introduction}

The celebrated Pollaczek-Khinchine (PK) formula is known in queuing
theory and risk theory in different forms. Following \cite{Asmussen2003},
under the safety loading condition, PK formula represents the cumulative
distribution function of the waiting time $W$ of an ordinary $M/G/1$
queue in terms of the arrival rate $\lambda,$ the cumulative distribution
function $F\left(x\right)$ of the service time $Y,$ and its mean
as
\begin{equation}
\mathsf{P}\left\{ W<x\right\} =\left(1-\rho\right)\sum_{n=0}^{\infty}\rho^{n}F_{0}^{*n}\left(x\right),\label{eq:1.2}
\end{equation}
where $\rho=\lambda\mathsf{E}Y,$ $F_{0}\left(x\right)=\left(\mathsf{E}Y\right)^{-1}\int_{0}^{x}\overline{F}(y)dy,$
$\overline{F}(y)=1-F\left(y\right)$, $F_{0}^{*n}$ denotes the convolution
of the distribution $F_{0}$ with itself. Following \cite{Asmussen2010},
formula (\ref{eq:1.2}) in risk theory is called the Beekman's convolution
formula and represents the survival probability for the classical
risk process (with $\rho=\lambda\mathsf{E}Y/c,$ where $c$ is the
premium rate).

Formula (\ref{eq:1.2}) is a special case of a more general result
on the distribution of the maximum of a spectrally positive L\'evy process.
If $X_{t},t\geq0$ is a L\'evy process with no negative jumps and the
cumulant function $k\left(r\right)=\ln\mathsf{E}e^{rX_{1}},$ then
the moment generating function for the absolute maximum $X^{+}=\sup_{0\leq t<\infty}X_{t}$
has the following representation (see, e.g., \cite{Kyprianou2006}):
\begin{equation}
\mathsf{E}e^{rX^{+}}=-\frac{rk'\left(0+\right)}{k\left(r\right)}.\label{eq:1.3}
\end{equation}
Inverting formula (\ref{eq:1.3}) with respect to $r,$ we can get
the generalization of formula (\ref{eq:1.2}) (see \cite{Kella2011}).
For the interpretation of the generalized PK for general perturbed
risk processes under the Cramer--Lundberg conditions, see \cite{Huzak2004}
and references therein.

In \cite{Husak2010a}, the analog of formula (\ref{eq:1.3}) was obtained
without specific restrictions on the distribution of negative jumps.
Using this analog, the generalization of~\eqref{eq:1.2} for risk
processes with exponentially distributed premiums (with parameter
$b>0$) can be represented in terms of the convolutions of $\overline{F}_{0}(x)=\overline{F}(x)+b\overline{\overline{F}}(x)$
(see~\cite{Husak2011engl} Table IV in appendix). For some alternative
representations of the supremum distribution for a L\'evy process, when
the L\'evy measure is nonzero on $\left(-\infty,0\right),$ see \cite{Kwasnicki2012,Chaumont2010,Kuznetsov}
and references therein.

In this article, we combine the results of \cite{Bratiychuk1990,Lewis2008,Husak2010a}
to get the L\'evy triplet for the supremum of a L\'evy process, when the
negative jumps have a matrix-exponential distribution. We consider
two special cases of matrix-exponential distributions, namely the
Erlang and hyperexponential ones, in detail.

\section{The unified form of the Pollaczek--Khinchine formula}
Consider a L\'evy process $X_{t},t\geq0$ with the cumulant function
\[
k\left(r\right)=a'r+\frac{\sigma^{2}}{2}r^{2}+\int_{-\infty}^{\infty}\left(e^{rx}-1-rxI_{\left\{ |x|\leq1\right\} }\right)\Pi\left(dx\right),
\]
where $a'$ and $\sigma$ are real constants, and $\Pi$ is a non-negative
measure defined on $R\backslash\{0\}$ such that $\int_{R}\max\left\{ x^{2},1\right\} \Pi\left(dx\right)<\infty.$
The components $\left\{ a',\sigma,\Pi\right\} $ are called a L\'evy
triplet. Hereinafter, we assume that $\gamma=\int_{-1}^{1}|x|\Pi\left(dx\right)<\infty.$
Then the cumulant function has the form
\begin{equation}
k\left(r\right)=ar+\frac{\sigma^{2}}{2}r^{2}+\int_{-\infty}^{\infty}\left(e^{rx}-1\right)\Pi\left(dx\right),\label{eq:2.1}
\end{equation}
where $a=a'-\gamma.$

We write $X_{t}^{+}$ and $X_{t}^{-}$ for the supremum and infimum
processes and $X^{+}$ and $X^{-}$ for the absolute extrema:
\[
X_{t}^{\pm}=\sup\left(\inf\right)_{0\leq t'\leq t}X_{t},X^{\pm}=\sup\left(\inf\right)_{0\leq t<\infty}X_{t}.
\]
Denote, by $\theta_{s},$ an exponentially distributed random variable
with parameter $s>0$: $\mathsf{P}\left\{ \theta_{s}>t\right\} =e^{-st},t>0,$
independent of $X_{t},$ and $\theta_{0}=\infty.$ Then $X_{\theta_{s}}$
is called a L\'evy process killed at the rate $s$ (see \cite{Bertoin1996}),
and its moment generating function is as follows:
\[
\mathsf{E}e^{rX_{\theta_{s}}}=\frac{s}{s-k\left(r\right)},\mathrm{Re}\left[r\right]=0.
\]
The moment generating functions of killed extrema $\mathsf{E}e^{rX_{\theta_{s}}^{\pm}}$
represent the Wiener--Hopf factors of $\mathsf{E}e^{rX_{\theta_{s}}},$
i.e.,
\[
\mathsf{E}e^{rX_{\theta_{s}}}=\mathsf{E}e^{rX_{\theta_{s}}^{+}}\mathsf{E}e^{rX_{\theta_{s}}^{-}}.
\]
Having a representation of one of the factors, we can get that for
another one.

In the general case, we can represent $\mathsf{E}e^{rX_{\theta_{s}}^{+}}$
in terms of the convolution of the cumulative distribution function
of the killed infimum $P_{-}\left(s,y\right)=\mathsf{P}\left\{ X_{\theta_{s}}^{-}<y\right\} $
with the integral transform of the measure $\Pi$: 
$\tilde{\Pi}\left(x,u\right)=\int_{x}^{\infty}e^{u\left(x-z\right)}\Pi\left(dz\right).$

In \cite{Husak2011engl} (see corollary 2.2), it was established that
if $\mathsf{D}X_{1}<\infty,\mathsf{E}X_{\theta_{s}}^{+}<\infty,$
then
\begin{equation}
\mathsf{E}e^{rX_{\theta_{s}}^{+}}=\left(1-r\left(A_{*}\left(s\right)+s^{-1}U\left(s,r\right)\right)\right)^{-1},\label{eq:2.3}
\end{equation}
where
\[
A_{*}\left(s\right)=\left\{ \begin{array}{lc}
\left(2s\right)^{-1}\sigma^{2}\left.\frac{\partial}{\partial y}P_{-}(s,y)\right|_{y=0} & \sigma>0,\\
s^{-1}\mathsf{P}\left\{ X_{\theta_{s}}^{-}=0\right\} \max\left\{ 0,a\right\}  & \sigma=0,
\end{array}\right.
\]
$U\left(s,r\right)=\int_{0}^{\infty}e^{rx}\int_{-\infty}^{0}\tilde{\Pi}\left(x-y,0\right)dP_{-}\left(s,y\right)dx.$

Integrating by parts, we get $U\left(s,r\right)=\int_{0}^{\infty}\left(e^{rx}-1\right)\int_{-\infty}^{0}{\Pi}\left(dx-y\right)dP_{-}\left(s,y\right).$
As was shown in~\cite{Husak2011engl}, if $\mu=\mathsf{E}X_{1}=k'\left(0\right)<0,$
then there exist $\lim_{s\rightarrow0}A_{*}\left(s\right)=a_{*}$
and
\begin{multline*}
\lim_{s\rightarrow0}s^{-1}U\left(s,r\right)=-\int_{0}^{\infty}\left(e^{rx}-1\right)\int_{-\infty}^{0}{\Pi}\left(dx-y\right)d\left(\int_{0}^{\infty}\mathsf{P}\left\{ X_{t}^{-}>y\right\} dt\right)=\\
=-\int_{0}^{\infty}\int_{-\infty}^{0}\left(e^{rx}-1\right){\Pi}\left(dx-y\right)d\mathsf{E}\tau^{-}(y),
\end{multline*}
where $\tau^{-}(y)=\inf\{t>0:X_{t}<y\},$ $y<0.$ Denote $p_{-}(0,y)=-\mathsf{E}\tau^{-}(y).$
Then we have the following statement. \begin{prp} \label{proposition1}
For a L\'evy process $X_{t},$ if $\mathsf{D}X_{1}<\infty,\mathsf{E}X_{\theta_{s}}^{+}<\infty,$
then
\begin{equation}
\mathsf{E}e^{rX_{\theta_{s}}^{+}}=\left(1-\left(rA_{*}\left(s\right)+\int_{0}^{\infty}\left(e^{rx}-1\right)\pi_{s}\left(dx\right)\right)\right)^{-1},\label{eq:2.3-1}
\end{equation}
where $\pi_{s}\left(dx\right)=\int_{-\infty}^{0}{\Pi}\left(dx-y\right)s^{-1}dP_{-}\left(s,y\right),x>0.$

If $\mu=\mathsf{E}X_{1}<0$ the moment generating function of absolute
supremum $X^{+}$ can be represented as the moment generating function
of the subordinator $X^{*}$ (a L\'evy process with bounded variation,
drift $a_{*}\geq0$ and measure $\Pi_{*}\left(dx\right)=\int_{-\infty}^{0}{\Pi}\left(dx-y\right)dp_{-}\left(0,y\right)$
is concentrated on $\left(0,\infty\right)$) killed at rate 1:
\begin{equation}
\mathsf{E}e^{rX^{+}}=\frac{1}{1-k_{*}\left(r\right)}=\mathsf{E}e^{rX_{\theta_{1}}^{*}}.\label{eq:2-10}
\end{equation}
\end{prp} 
This proposition offers the following scheme for finding
the m.g.f. for the supremum of a L\'evy process. First, we should find
the distribution of the infimum killed at a rate $s$ and its convolution
with $\tilde{\Pi}\left(x,0\right)=\int_{x}^{\infty}\Pi\left(dx\right).$
Then, under the condition $\mu<0$ after the limit transition as $s$
tends to 0, we can get the L\'evy triplet of $X^{+}.$

We call formula (\ref{eq:2-10}) the unified Pollaczek--Khinchine
formula. To justify the name, we follow the approach from \cite{Kella2011}.

Assume $\Pi_{*}\left(0,+\infty\right)<\infty$ and $a_{*}>0.$ Then,
with the use of the notation $\rho=1-\left(1+\Pi_{*}\left(0,+\infty\right)\right)^{-1},$
$c_{*}=\left(a_{*}\left(1-\rho\right)\right)^{-1},$ and $\hat{\Pi}_{*}\left(r\right)=\int_{\left(0,\infty\right)}e^{rx}\Pi_{*}\left(dx\right)/\Pi_{*}\left(0,+\infty\right),$
relation (\ref{eq:2-10}) can be rewritten as
\begin{multline}
\mathsf{E}e^{rX^{+}}=\frac{1-\rho}{1-\frac{r}{c_{*}}-
\rho\hat{\Pi}_{*}\left(r\right)}=
\frac{c_{*}}{c_{*}-r}\frac{1-\rho}{1-\frac{c_{*}}{c_{*}-r}\rho\hat{\Pi}_{*}\left(r\right)}=\\
=\left(1-\rho\right)
\sum_{n=0}^{\infty}\left(\rho\hat{\Pi}_{*}\left(r\right)\right)^{n}\left(\frac{c_{*}}{c_{*}-r}\right)^{n+1}.\label{eq:2.10a}
\end{multline}
If $X_{t}$ is a compound Poisson process with drift $-1,$ positive
jumps, and $\mathsf{E}X_{1}<0,$ then formula (\ref{eq:2.10a}) can
be reduced to $\mathsf{E}e^{rX^{+}}=\left(1-\rho\right)\sum_{n=0}^{\infty}\left(\rho\int_{0}^{\infty}e^{rx}dF_{0}(x)\right)^{n}$
with $dF_{0}(x)=\left(\Pi_{*}\left(0,+\infty\right)\right)^{-1}\Pi_{*}\left(dx\right).$
After the inversion with respect to $r,$ we get the classical PK
formula (\ref{eq:1.2}).

\section{Matrix-exponential negative jumps}
To be more precise, we use some additional conditions on $\Pi\left(dx\right)$
if $x<0.$ Assume that $\int_{\left(-\infty,0\right)}\Pi\left(dx\right)=\lambda_{-}<\infty,$
and the negative jumps have a distribution with rational characteristic
function (or matrix-exponential distribution). Then, following \cite{Lewis2008},
the L\'evy measure for negative $x$ has the representation
\[
\Pi\left(-\infty,x\right)=\sum_{i=1}^{m}\sum_{j=0}^{k_{i}}\int_{-\infty}^{x}a_{j}^{\left(i\right)}\left(-y\right)^{j}e^{b_{i}y}dy;\qquad x<0,
\]
where $\mathrm{Re}[b_{m}]\geq\ldots\geq\mathrm{Re}[b_{2}]\geq b_{1}>0.$

We study separately two cases:
\begin{description}
\item [{(NS)}] If $\sigma>0$ or $\sigma=0,a<0.$
\item [{(S)}] If $\sigma=0,$ $a\geq0.$
\end{description}
In~\cite{Lewis2008}, it was established that the cumulant equation
$k\left(r\right)=s$ have $N$ roots $-r_{i}\left(s\right)$ in the
half-plane $\mathrm{Re}\left[r\right]<0,$ and $-r_{1}\left(s\right)$
is the unique root in $\left[-b_{1},0\right],$ where
\[
N=\left\{ \begin{array}{l}
\sum_{i=1}^{m}k_{i}+m+1,\mbox{ in case (NS),}\\
\sum_{i=1}^{m}k_{i}+m,\mbox{ in case (S}).
\end{array}\right.
\]
We assume that the roots are ordered in the ascending order of their
real parts
\[
\mathrm{Re}\left[r_{N}\left(s\right)\right]\geq\ldots\geq\mathrm{Re}\left[r_{2}\left(s\right)\right]>r_{1}\left(s\right).
\]
By~\cite{Lewis2008} (see also formula (2.25) \cite{Bratiychuk1990}),
\begin{equation}
\mathsf{E}e^{rX_{\theta_{s}}^{-}}={\displaystyle \frac{\prod_{i=1}^{N}r_{i}\left(s\right)\prod_{i=1}^{m}\left(r+b_{i}\right)^{k_{i}+1}}{\prod_{i=1}^{m}b_{i}^{k_{i}+1}\prod_{i=1}^{N}\left(r+r_{i}\left(s\right)\right)}},\quad\mathrm{Re}\left[r\right]=0.\label{eq:2.5-1}
\end{equation}

Let the cumulant equation have $l$ distinct roots $r_{k}\left(s\right)$
with multiplicity $n_{k},$ $k=\overline{1,l}.$ Then the density
of $X_{\theta_{s}}^{-}$ has the representation
\begin{equation}
P'_{-}\left(s,y\right)=\frac{\partial}{\partial y}P_{-}(s,y)=D_{-}\left(s,y\right)+{\displaystyle \frac{\prod_{i=1}^{N}r_{i}\left(s\right)}{\prod_{i=1}^{m}b_{i}^{k_{i}+1}}}\sum_{k=1}^{l}e^{r_{k}\left(s\right)y}\sum_{j=1}^{n_{k}}\frac{A_{k,j}\left(s\right)}{\left(j-1\right)!}y^{j-1},y<0,\label{eq:2-22}
\end{equation}
where
\[
A_{k,j}\left(s\right)=\frac{1}{\left(n_{k}-j\right)!}\lim_{r\rightarrow-r_{k}\left(s\right)}\left[\frac{d^{\left(n_{k}-j\right)}}{dr^{\left(n_{k}-j\right)}}{\displaystyle \left(\frac{\left(r+r_{k}\left(s\right)\right)^{n_{k}}\prod_{i=1}^{m}\left(r+b_{i}\right)^{k_{i}+1}}{\prod_{i=1}^{N}\left(r+r_{i}\left(s\right)\right)}\right)}\right],
\]
\[
D_{-}\left(s,x\right)=\left\{ \begin{array}{l}
0,\mbox{ in case (NS),}\\
\frac{\prod_{i=1}^{N}r_{i}\left(s\right)}{\prod_{i=1}^{m}b_{i}^{k_{i}+1}}\delta\left(x\right),\mbox{ in case (S}),
\end{array}\right.
\]
and $\delta\left(x\right)$ is the Dirac delta function.

It was shown in~\cite{Lewis2008} that if $\mu=\mathsf{E}X_{1}<0,$
then, as $s\rightarrow0,$ $s^{-1}r_{1}\left(s\right)\rightarrow|\mu|^{-1},$
$r_{i}\left(s\right)\rightarrow r_{i},\mathrm{Re}\left[r_{i}\right]>0,$
$2\leq i\leq N.$ Hence,
\[
{\displaystyle \frac{\prod_{i=1}^{N}r_{i}\left(s\right)}{\prod_{i=1}^{m}b_{i}^{k_{i}+1}}}\rightarrow{\displaystyle \frac{\prod_{i=2}^{N}r_{i}}{|\mu|\prod_{i=1}^{m}b_{i}^{k_{i}+1}}}
\]
and
\begin{multline}
p'_{-}\left(0,y\right)=\frac{\partial}{\partial y}p_{-}\left(0,y\right)
=D_{-}\left(y\right)+{\displaystyle \frac{\prod_{i=2}^{N}r_{i}}{|\mu|\prod_{i=1}^{m}b_{i}^{k_{i}+1}}}\sum_{k=2}^{l}e^{r_{k}y}\sum_{j=1}^{n_{k}}\frac{A_{k,j}\left(0\right)}{\left(j-1\right)!}y^{j-1},y<0,\label{eq:2-23}
\end{multline}
\[
D_{-}\left(y\right)=\left\{ \begin{array}{l}
0,\mbox{ in case (NS),}\\
\frac{|\mu|^{-1}\prod_{i=2}^{N}r_{i}}{\prod_{i=1}^{m}b_{i}^{k_{i}+1}}\delta\left(x\right),\mbox{ in case (S}).
\end{array}\right.
\]
If the multiplicity of a real root $r_{k}\left(s\right)$ is 1, then
the corresponding addends in $P'_{-}\left(s,y\right)$ are simplified
to the exponential function $A_{k}\left(s\right)e^{r_{k}(s)y}.$ If
$r_{k}\left(s\right)=v_{k}\left(s\right)+\imath w_{k}\left(s\right)$
is a complex root of multiplicity 1, then the corresponding addend
in $P'_{-}\left(s,y\right)$ reduces to the product of exponential
and cosine functions $A_{k}\left(s\right)e^{v_{k}\left(s\right)y}\cos\left(w_{k}\left(s\right)y+\varphi_{k}\left(s\right)\right),$
where $A_{k}\left(s\right)=2|\mathrm{res}\left(r{}_{k}\left(s\right)\right)|,$
$\varphi_{k}\left(s\right)=\mathrm{arg}\left(\mathrm{res}\left(r{}_{k}\left(s\right)\right)\right),$
$\mathrm{res}\left(r_{k}\left(s\right)\right)=\left[\left(r+r_{k}\left(s\right)\right)\mathsf{E}e^{rX_{\theta_{s}}^{-}}\right]_{r\rightarrow-r_{k}\left(s\right)}.$

\section{Special cases}
Further, we consider several examples, when the negative jumps have
the Erlang and hyperexponential distributions. In both cases in the
half-plane $\mathrm{Re}\left[r\right]<0,$ the cumulant equation $k\left(r\right)=s$
has $N=d+1$ roots in case (NS) and $N=d$ roots in case (S) with
negative real parts.

\subsection{Hyperexponentially distributed negative jumps}
If a L\'evy process $X_{t}$ have hyperexponential negative jumps, then
all negative roots of the cumulant equation are distinct and real
(see, e.g., \cite{Kou2011} or \cite{Mordecki2002a},) and the distribution
of the killed infimum $X_{\theta_{s}}^{-}$ $ $is mixed exponential.
\begin{cor}\label{cor1} Let a process $X_{t}$ have hyperexponentially
distributed negative jumps. If $\mu=\mathsf{E}X_{1}<0,$ then the
absolute maximum $X^{+}$ is distributed as the subordinator with
L\'evy triplet $\left(a_{*},0,\Pi_{*}\left(dx\right)\right)$ killed
rate 1, where, in case (NS),
\begin{equation}
a_{*}=\frac{\sigma^{2}}{2|\mu|}\left(1-\sum_{k=2}^{d+1}\frac{\prod_{i=1}^{d}\left(1-r_{k}/b_{i}\right)}{\prod_{i=2,i\neq k}^{N}\left(1-r_{k}/r_{i}\right)}\right),\label{eq:2-26}
\end{equation}
\[
\Pi_{*}\left(dx\right)=|\mu|^{-1}\left(\tilde{\Pi}\left(x,0\right)-\sum_{k=2}^{d+1}\frac{\prod_{i=1}^{d}\left(1-r_{k}/b_{i}\right)}{\prod_{i=2,i\neq k}^{d+1}\left(1-r_{k}/r_{i}\right)}\tilde{\Pi}\left(x,r_{k}\right)\right)dx,
\]
and, in case (S),
\begin{equation}
a_{*}=a|b_{1}\mu|^{-1}\prod_{i=2}^{d}\left(r_{i}/b_{i}\right),\label{eq:2-27}
\end{equation}
\[
\Pi_{*}\left(dx\right)=\frac{1}{|\mu|}\left(\frac{1}{b_{1}}\prod_{i=2}^{d}\frac{r_{i}}{b_{i}}{\Pi}\left(dx\right)+\left(\tilde{\Pi}\left(x,0\right)-\sum_{k=2}^{d}\frac{\prod_{i=1}^{d}\left(1-r_{k}/b_{i}\right)}{\prod_{i=2,i\neq k}^{d}\left(1-r_{k}/r_{i}\right)}\tilde{\Pi}\left(x,r_{k}\right)\right)dx\right).
\]
\end{cor}
\begin{proof} As all $r_{k}\left(s\right)$ are distinct
and real, formula (\ref{eq:2-22}) yields
\begin{equation}
P'_{-}\left(s,y\right)=D_{-}\left(s,y\right)+{\displaystyle \frac{\prod_{i=1}^{N}r_{i}\left(s\right)}{\prod_{i=1}^{d}b_{i}}}\sum_{k=1}^{N}\frac{\prod_{i=1}^{d}\left(b_{i}-r_{k}\left(s\right)\right)}{\prod_{i=1,i\neq k}^{N}\left(r_{i}\left(s\right)-r_{k}\left(s\right)\right)}e^{r_{k}\left(s\right)y},y<0,\label{eq:2-24}
\end{equation}
where
\[
D_{-}\left(s,x\right)=\left\{ \begin{array}{l}
0,\mbox{ in case (NS),}\\
\prod_{i=1}^{d}\left(r_{i}\left(s\right)/b_{i}\right)\delta\left(x\right),\mbox{ in case (S}).
\end{array}\right.
\]
If $\mu<0,$ then
\begin{equation}
p'_{-}\left(0,y\right)=D_{-}\left(y\right)+|\mu|^{-1}\left(1-\sum_{k=2}^{N}\frac{\prod_{i=1}^{d}\left(1-r_{k}/b_{i}\right)}{\prod_{i=2,i\neq k}^{N}\left(1-r_{k}/r_{i}\right)}e^{r_{k}y}\right),y<0,\label{eq:2-25}
\end{equation}
where
\[
D_{-}\left(x\right)=\left\{ \begin{array}{l}
0,\mbox{ in case (NS),}\\
|b_{1}\mu|^{-1}\prod_{i=2}^{d}\left(r_{i}/b_{i}\right)\delta\left(x\right),\mbox{ in case (S}).
\end{array}\right.
\]
Combining formula (\ref{eq:2-25}) with Proposition \ref{proposition1},
we get (\ref{eq:2-26}) and (\ref{eq:2-27}). Note that $\tilde{\Pi}(x,0)$
is just $\int_{x}^{\infty}\Pi(dy),x>0.$\end{proof} \begin{exm}
If $d=2$ in Corollary~\ref{cor1}, then the parameters of $X^{*}$
in case (NS) are
\[
a_{*}=\frac{\sigma^{2}}{2|\mu|}\frac{r_{2}r_{3}}{b_{1}b_{2}},
\]
\begin{multline*}
\Pi_{*}\left(dx\right)=\frac{1}{|\mu|}\left[\tilde{\Pi}\left(x,0\right)+\frac{r_{3}\left(r_{2}-b_{1}\right)\left(b_{2}-r_{2}\right)}{b_{1}b_{2}\left(r_{3}-r_{2}\right)}\tilde{\Pi}\left(x,r_{2}\right)+\right.\\
\left.+\frac{r_{2}\left(r_{3}-b_{1}\right)\left(r_{3}-b_{2}\right)}{b_{1}b_{2}\left(r_{3}-r_{2}\right)}\tilde{\Pi}\left(x,r_{3}\right)\right]dx.
\end{multline*}
In case (S), we have
\[
a_{*}=\frac{ar_{2}}{|\mu|b_{1}b_{2}},
\]
\[
\Pi_{*}\left(dx\right)=\frac{r_{2}}{|\mu|b_{1}b_{2}}\left[\Pi\left(dx\right)+\left(\frac{b_{1}b_{2}}{r_{2}}\tilde{\Pi}\left(x,0\right)+\frac{\left(r_{2}-b_{1}\right)\left(b_{2}-r_{2}\right)}{r_{2}}\tilde{\Pi}\left(x,r_{2}\right)\right)dx\right].
\]
\end{exm}

\subsection{Erlang-distributed negative jumps}
If the negative jumps have the Erlang distribution, then we will use
a matrix representation alternative to (\ref{eq:2-22}) for the inversion
of m.g.f. of $X_{\theta_{s}}^{-}$ (for details, see \cite{Asmussen2010}):
if
\[
\int_{-\infty}^{0}e^{rx}f\left(x\right)dx=\frac{\sum_{k=1}^{n}\alpha_{k}r^{k-1}}{\sum_{k=1}^{n}t_{k}r^{n-k}+r^{n}},
\]
then
\begin{equation}
f\left(x\right)=\boldsymbol{\alpha}e^{\mathbf{T}x}\mathbf{t},\;\; x<0,\label{eq:2-27a}
\end{equation}
where $\boldsymbol{\alpha}=\left(\alpha_{1},\ldots,\alpha_{n}\right),$
$\mathbf{t}=\left(0,\ldots,0,1\right)^{\top},$ $\mathbf{T}=\left(\begin{array}{cccc}
0 & -1 & \ldots & 0\\
\vdots & \vdots & \ddots & \vdots\\
0 & 0 & \ldots & -1\\
t_{n} & t_{n-1} & \ldots & t_{1}
\end{array}\right).$ \begin{cor}\label{cor2} Let the process $X_{t}$ have Erlang-distributed
negative jumps. If $\mu=\mathsf{E}X_{1}<0,$ then the absolute maximum
$X^{+}$ is distributed as the subordinator with L\'evy triplet $\left(a_{*},0,\Pi_{*}\left(dx\right)\right)$
killed rate 1, where, in case (NS), we have
\[
a_{*}=\frac{\sigma^{2}}{2|\mu|}\prod_{i=2}^{d+1}\left(\frac{r_{i}}{b}\right),
\]
\begin{equation}
\Pi_{*}\left(dx\right)=|\mu|^{-1}\prod_{i=2}^{d+1}\left(\frac{r_{i}}{b}\right)
\boldsymbol{\alpha}\int_{x}^{\infty}e^{\mathbf{T}\left(x-y\right)}\Pi\left(dy\right)\mathbf{t}dx,\label{eq:2-30}
\end{equation}
$\alpha_{k}=\left(_{k-1}^{d}\right)b^{d+1-k},$ $t_{k}=\sum_{1\leq i_{1}<\ldots<i_{k}\leq d+1}r_{i_{1}}\ldots r_{i_{k}},$
$\left(_{k-1}^{d}\right)$ are binomial coefficients, $k=\overline{1,d+1}.$

In case (S),
\[
a_{*}=a|b\mu|^{-1}\prod_{i=2}^{d}\left(r_{i}/b\right),
\]
\begin{equation}
\Pi_{*}\left(dx\right)=|b\mu|^{-1}\prod_{i=2}^{d}\left(\frac{r_{i}}{b}\right)\left(\Pi\left(dx\right)+\boldsymbol{\alpha}\int_{x}^{\infty}e^{\mathbf{T}\left(x-y\right)}\Pi\left(dy\right)\mathbf{t}dx\right),\label{eq:2-31}
\end{equation}
$\alpha_{k}=\left(_{k-1}^{d}\right)b^{d+1-k}-t_{d+1-k},$ $t_{k}=\sum_{1\leq i_{1}<\ldots<i_{k}\leq d}r_{i_{1}}\ldots r_{i_{k}},$
$k=\overline{1,d}.$ \end{cor} \begin{proof} Using the inversion
formula (\ref{eq:2-27a}) from (\ref{eq:2.5-1}), we deduce in case
(NS) that
\begin{equation}
P'_{-}\left(s,y\right)=r_{1}\left(s\right)\prod_{i=2}^{d+1}\frac{r_{i}\left(s\right)}{b}
\boldsymbol{\alpha}e^{\mathbf{T}\left(s\right)y}\mathbf{t},y<0.\label{eq:2-28}
\end{equation}
In case (S),
\begin{equation}
P'_{-}\left(s,y\right)=r_{1}\left(s\right)
\frac{\prod_{i=2}^{d}r_{i}\left(s\right)}{b^{d}}\left(\delta\left(y\right)
+\boldsymbol{\alpha}e^{\mathbf{T}\left(s\right)y}\mathbf{t}\right),y<0.\label{eq:2-29}
\end{equation}
As $s$ approaches 0, we can get the representation for $p'_{-}\left(0,y\right)$
from (\ref{eq:2-28}) and (\ref{eq:2-29}). From Proposition~\ref{proposition1},
we deduce (\ref{eq:2-30}) and (\ref{eq:2-31}).\end{proof} To expand
the matrix exponents in (\ref{eq:2-30}) and (\ref{eq:2-31}), we
should know the values of their multiplicities and whether the roots
$r_{i}$ are real. \begin{exm} Assume that $d=2$ in Corollary~\ref{cor2}.
In case (S) as $s$ approaches 0: $r_{1}\left(s\right)\rightarrow0,r_{2}\left(s\right)\rightarrow r_{2}>0.$
Then $\boldsymbol{\alpha}=\left(b^{2},2b-r_{2}\right),\mathbf{T}=\left(\begin{array}{cc}
0 & -1\\
0 & r_{2}
\end{array}\right).$ Hence,
\[
p'_{-}\left(0,y\right)=\frac{r_{2}}{|\mu|b^{2}}\left(\delta\left(x\right)
+\frac{b^{2}}{r_{2}}-\frac{\left(b-r_{2}\right)^{2}}{r_{2}}e^{r_{2}y}\right)
\]
and
\[
a_{*}=\frac{ar_{2}}{|\mu|b^{2}},
\]
\begin{multline*}
\Pi_{*}\left(dx\right)=\frac{r_{2}}{|\mu|b^{2}}\left(\Pi\left(dx\right)
+\boldsymbol{\alpha}\int_{x}^{\infty}e^{\mathbf{T}\left(x-y\right)}\Pi\left(dy\right)\mathbf{t}dx\right)=\\
=\frac{r_{2}}{|\mu|b^{2}}\left(\Pi\left(dx\right)
+\left(\frac{b^{2}}{r_{2}}\tilde{\Pi}\left(x,0\right)-\frac{\left(b-r_{2}\right)^{2}}{r_{2}}\tilde{\Pi}\left(x,r_{2}\right)\right)dx\right).
\end{multline*}

In case (NS) as $s$ approaches 0: $r_{1}\left(s\right)\rightarrow0,r_{2}\left(s\right)\rightarrow r_{2},$
$r_{3}\left(s\right)\rightarrow r_{3}.$ Then
\[
a_{*}=\frac{\sigma^{2}}{2|\mu|}\frac{r_{2}r_{3}}{b^{2}},
\]
$\boldsymbol{\alpha}=\left(b^{2},2b,1\right),\mathbf{T}=\left(\begin{array}{ccc}
0 & -1 & 0\\
0 & 0 & -1\\
0 & r_{2}r_{3} & \left(r_{2}+r_{3}\right)
\end{array}\right).$ In this case, $\Pi_{*}\left(dx\right)$ can contain, along with the
integral transform $\tilde{\Pi}\left(x,u\right),$ also the integral
transform
\[
B\left(x,u\right)=\int_{x}^{\infty}\left(x-z\right)e^{u\left(x-z\right)}\Pi\left(dz\right)
\]
and integral transforms $C_{1}\left(x,v,w\right)=\int_{x}^{\infty}e^{v\left(x-z\right)}\cos\left(w\left(x-z\right)\right)\Pi\left(dz\right),$
$C_{2}\left(x,v,w\right)=\int_{x}^{\infty}e^{v\left(x-z\right)}\sin\left(w\left(x-z\right)\right)\Pi\left(dz\right).$
More precisely, we have 3 possibilities:

1) $r_{2,3}$ are real and distinct, then
\[
\boldsymbol{\alpha}e^{\mathbf{T}x}\mathbf{t}=\frac{b^{2}}{r_{2}r_{3}}+\frac{e^{r_{3}x}r_{2}\left(b-r_{3}\right)^{2}-e^{r_{2}x}r_{3}\left(b-r_{2}\right)^{2}}{r_{2}r_{3}\left(r_{3}-r_{2}\right)},
\]
so $\Pi_{*}\left(dx\right)=|\mu|^{-1}\left(\tilde{\Pi}\left(x,0\right)-\frac{r_{3}\left(b-r_{2}\right)^{2}}{b^{2}\left(r_{3}-r_{2}\right)}\tilde{\Pi}\left(x,r_{2}\right)+\frac{r_{2}\left(b-r_{3}\right)^{2}}{b^{2}\left(r_{3}-r_{2}\right)}\tilde{\Pi}\left(x,r_{3}\right)\right)dx.$
Here, we get the same formula as for the hyperexponential distribution
with $b_{1}=b_{2}=b.$

2) $r_{2,3}$ are real and equal, then
\[
\boldsymbol{\alpha}e^{\mathbf{T}x}\mathbf{t}=\frac{b^{2}}{r_{2}^{2}}\left(1+e^{r_{2}x}\left(\frac{r_{2}^{2}-b^{2}}{b^{2}}+\frac{\left(b-r_{2}\right)^{2}r_{2}}{b^{2}}x\right)\right),
\]
$\Pi_{*}\left(dx\right)=|\mu|^{-1}\left(\tilde{\Pi}\left(x,0\right)-\frac{r_{2}^{2}-b^{2}}{b^{2}}\tilde{\Pi}\left(x,r_{2}\right)+\frac{\left(b-r_{2}\right)^{2}}{b^{2}}B\left(x,r_{2}\right)\right)dx.$

3) $r_{2,3}$ are complex-valued: $r_{2,3}=v\pm\imath w,$ then
\[
\boldsymbol{\alpha}e^{\mathbf{T}x}\mathbf{t}=\frac{b^{2}}{v^{2}+w^{2}}+
\]
\[
+e^{vx}\left(\frac{\left(v^{2}+w^{2}-b^{2}\right)}{v^{2}+w^{2}}\cos(wx)+\frac{\left(v+b\left(\frac{bv}{v^{2}+w^{2}}-2\right)\right)}{w}\sin(wx)\right),
\]
$\Pi_{*}\left(dx\right)=|\mu|^{-1}\left(\tilde{\Pi}\left(x,0\right)dx+\frac{|r_{2}|^{2}-b^{2}}{b^{2}}C_{1}\left(x,v,w\right)+\frac{b^{2}v+|r_{2}|^{2}\left(v-2b\right)}{b^{2}w}C_{2}\left(x,v,w\right)\right)dx.$
\end{exm} \begin{exm} In conclusion, we consider a numerical example,
when a process $X_{t}$ have L\'evy triplet $\left(a,\sigma,\Pi\left(dx\right)\right),$
where
\[
\Pi\left(dx\right)=\begin{cases}
\lambda_{+}f_{+}\left(x\right) & x>0,\\
\lambda_{-}f_{-}\left(x\right) & x<0,
\end{cases}=\begin{cases}
{\displaystyle \lambda_{+}\frac{2\beta}{\pi}e^{-\frac{x^{2}\beta^{2}}{\pi}}} & x>0,\\
\lambda_{-}\left(1+\frac{1}{4\pi^{2}}\right)\left(1-\cos\left(2\pi x\right)\right)e^{x} & x<0,
\end{cases}
\]
i.e., the positive jumps are half-normal$\left(\beta\right)$ distributed,
and the negative jumps are matrix-exponential distributed with $b_{1}=1,$
$b_{2,3}=1\pm2\pi\imath$ (a special case of the matrix-exponential
distribution, which is not a phase-type distribution \cite{Asmussen2010}).
The cumulant of $X_{t}$ is
\begin{multline*}
k\left(r\right)=ar+\frac{\sigma^{2}}{2}r^{2}+\lambda_{+}\left(e^{\frac{\pi r^{2}}{4\beta^{2}}}\left(1+\mathrm{erf}\left(\frac{\sqrt{\pi}r}{2\beta}\right)\right)-1\right)+\\
+\lambda_{-}\left(\frac{1+4\pi^{2}}{\left(1+r\right)\left(4\pi^{2}+\left(1+r\right)^{2}\right)}-1\right),
\end{multline*}
where $\mathrm{erf}\left(x\right)=\frac{2}{\sqrt{\pi}}\int_{0}^{z}e^{-t^{2}}dt.$

For $a=0.2,$ $\sigma=2,$ $\lambda_{+}=2,$ $\beta=1,$ $\lambda_{-}=4,$
$m=\mathsf{E}X_{1}=\frac{49+36\pi^{2}}{-5-20\pi^{2}}<0,$ the cumulant
equation $k\left(r\right)=0$ has the roots $-r_{1}=0,$ $-r_{2,3}\approx1.023\pm6.290\imath,$
$-r_{4}\approx2.159$ in the half-plane $\mathrm{Re}\left[r\right]\leq0$
and, by formula (\ref{eq:2-23}),
\[
p'_{-}\left(0,x\right)\approx0.501+0.582e^{2.159x}+e^{1.023x}\left(0.002\cos\left(6.290x\right)+0.008\sin\left(6.290x\right)\right).
\]
Hence, $a_{*}=\frac{\sigma^{2}}{2}p'_{-}\left(0,0\right)\approx2.169$
and $\Pi_{*}\left(dx\right)=\int_{x}^{\infty}p'_{-}\left(0,x-y\right)\lambda_{+}f_{+}\left(y\right)dy$
can be expressed in terms of $\mathrm{erf}\left(x\right).$ Using
formula (\ref{eq:2.10a}), we can represent $X^{+}$ as the sum of
a geometrically distributed number $\nu\sim\mathrm{G}\left(1-\rho\right),$
exponentially distributed random variables $\xi_{0},\ldots,\xi_{n},\ldots\sim\mathrm{Exp}\left(c_{*}\right),$
and random variables $\eta_{1},\ldots,\eta_{n},\ldots\sim F{}_{0},$
where $1-\rho\approx0.418,$ $c_{*}\approx1.104,$ $F'_{0}(x)dx=\left(\Pi_{*}\left(0,+\infty\right)\right)^{-1}\Pi_{*}\left(dx\right)$
(see Fig. 1). That is,
\[
X^{+}=\xi_{0}+\sum_{n=1}^{\nu}\left(\xi_{n}+\eta_{n}\right).
\]
\begin{figure}\label{pic1}
\centering \includegraphics[scale=0.7]{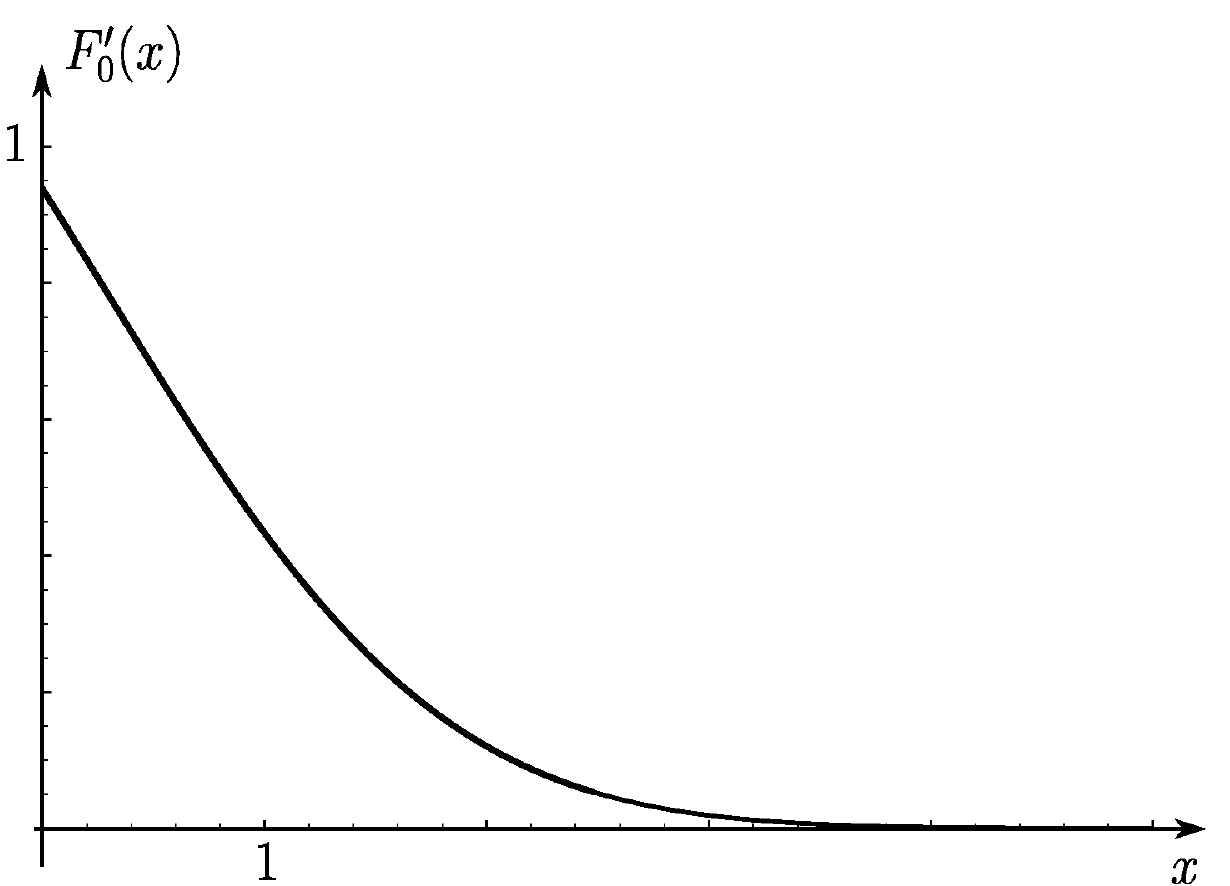} \caption{Density of $\eta_{i}$}
\end{figure}
\end{exm}

\end{document}